\documentclass[10pt]{svmult}
\usepackage{amssymb,amsmath,amscd}
\usepackage[T1]{fontenc}

\usepackage{mathrsfs}
\usepackage{color}
\usepackage{epsfig}

\usepackage[all]{xy}

\newcommand{\ZZ}{\mathbb{Z}}

\newcommand{\RR}{\mathbb{R}}
\newcommand{\CC}{\mathbb{C}}

\newcommand{\NN}{\mathbb{N}}
\newcommand{\HH}{\mathbb{H}}
\newcommand{\FF}{\mathbb{F}}

\newcommand{\PP}{\mathbb{P}}

\newcommand{\cM}{\mathcal{M}}
\newcommand{\cS}{\mathcal{S}}
\newcommand{\cO}{\mathcal{O}}
\newcommand{\cQM}{\mathcal{QM}}

\newcommand{\idx}[1]{\index{#1}{\em #1}}

\renewcommand{\hat}{\widehat}

\DeclareMathOperator{\Aut}{Aut}

\DeclareMathOperator{\im}{\mathfrak{Im}}

\makeindex

\begin{document}

\title{Introduction to Modular Forms}
\author{Simon C. F. Rose}
\institute{Field's Institute, Canada}

\maketitle

\abstract{We introduce the notion of modular forms, focusing primarily on the group \(PSL_2\ZZ\). We further introduce quasi-modular forms, as well as discuss their relation to physics and their applications in a variety of enumerative problems. These notes are based on a lecture given at the Field's Institute during the thematic program on Calabi-Yau Varieties: Arithmetic, Geometry, and Physics.}

\section{Introduction}

The goal of this chapter is to introduce a particular class of functions called {\em modular forms}. It should be remarked that there are many ways of looking at these functions; for the purpose of these notes we will focus on considering them as a certain type of generating function with particularly interesting coefficients.

That said, this is a rather myopic view. The theory of modular forms is much richer and more interesting than that, ranging widely through the fields of algebra, analysis, number theory, and geometry. A good reference to read further and to learn more would be the book \cite{diamond2005a}.

\section{Basic definitions}

We begin with the following setup. We first note that the group \(PSL_2\ZZ\) acts naturally on the upper-half plane \(\HH = \{\tau \in \CC \mid \im \tau > 0\}\) via the action
\[
\begin{pmatrix} a & b \\ c & d \end{pmatrix} \cdot \tau = \frac{a\tau + b}{c\tau + d}.
\]

\begin{remark}
It should be noted that this is simply a restriction of the usual action of \(PSL_2\CC\) on \(\PP^1\) to the subgroup of integer matrices with determinant 1. Since this preserves the real line, and due to the condition on the determinant, it also acts on \(\HH\). It is then easy to check that these two actions are the same.
\end{remark}

\begin{definition}\label{mod_def}
Let \(\Gamma\) be a finite index subgroup of \(PSL_2\ZZ\) (which will satisfy some conditions which will be discussed later). Then we say that a holomorphic function \(f : \HH \to \CC\) is a \idx{modular form} of weight \(k\) for the group \(\Gamma\) if it satisfies the transformation law
\[
f(\gamma \tau) = (c\tau + d)^k f(\tau)
\]
for all \(\gamma = \begin{pmatrix} a & b \\ c & d \end{pmatrix} \in \Gamma\), and if it is holomorphic at infinity; that is, if \(\lim_{\tau \to i\infty}f(\tau)\) is finite.
\end{definition}

A few remarks are in order.

\begin{remark}
We first note that if \(\Gamma = PSL_2\ZZ\), then the only modular form of odd weight is \(f(\tau) = 0\). More generally, if the matrix
\[
\begin{pmatrix}-1 & 0 \\ 0 & -1\end{pmatrix}
\]
is an element of the group \(\Gamma\), then this is true. This is since:
\[
f(\tau) = f(\gamma \tau) = (-1)^k f(\tau)
\]
must be true for all \(\tau \in \HH\), which can only hold if \(k\) is even.
\end{remark}

\begin{remark}
As written, this definition might seem unmotivated. It says that modular forms are certain functions are those which behave in a particular way under a certain group action, which is not particularly enlightening.

A more geometric way to regard these is as follows: modular forms are actually sections of line bundles over the stack quotient \([\HH/\Gamma]\). Consider the action of \(\Gamma\) on \(\HH \times \CC\) given by
\[
\begin{pmatrix} a & b \\ c & d \end{pmatrix} \cdot (\tau, z) = \Big(\frac{a\tau + b}{c\tau + d}, (c \tau + d)^k z\Big).
\]
Then there is a \(\Gamma\)-equivariant projection \(\HH \times \CC \to \HH\). This lets us regard the stack \([\HH \times \CC / \Gamma]\) as a line bundle over the stack \([\HH / \Gamma]\) as claimed.
\end{remark}

So far, the only examples of modular forms that we have are constant functions \(f(\tau) = C\). Fortunately, there are many more modular forms than just these functions.

\begin{definition}
Let \(k\) be a positive integer. Define the \idx{Eisenstein series} of weight \(2k\) to be the series
\[
G_{2k}(\tau) = \sum_{(m, n) \neq (0, 0)} \frac{1}{(m\tau + n)^{2k}}.
\]
\end{definition}

We begin by proving that this is modular of weight \(2k\).

\begin{theorem}
For \(k > 1\), the function \(G_{2k}(\tau)\) is modular of weight \(2k\).
\end{theorem}

\begin{proof}
We begin by noting that it suffices to check the transformation law for the generators of the group \(\Gamma\). For the case of \(PSL_2\ZZ\), this group is generated by the matrices
\[
\begin{pmatrix}1 & 1 \\ 0 & 1 \end{pmatrix} \qquad\text{and}\qquad \begin{pmatrix}0 & -1 \\ 1 & 0 \end{pmatrix},
\]
or equivalently the transformations
\[
\tau \mapsto \tau + 1 \qquad\text{and}\qquad \tau \mapsto -\frac{1}{\tau}.
\]

The first of these is easy to see:
\begin{align*}
G_{2k}(\tau + 1) &= \sum_{(m, n) \neq (0, 0)} \frac{1}{\big(m(\tau+1) + n\big)^{2k}}\\
 &= \sum_{(m, n) \neq (0, 0)} \frac{1}{\big(m\tau + (n+m)\big)^{2k}},
\end{align*}
which is clearly the same as \(G_{2k}(\tau)\).

For the second, we have that
\begin{align*}
G_{2k}(-1/\tau) &= \sum_{(m, n) \neq (0, 0)} \frac{1}{(-m/\tau + n)^{2k}} \\
&= \sum_{(m, n) \neq (0, 0)} \tau^{2k}\frac{1}{(-m + n\tau)^{2k}} \\
&= \tau^{2k} G_{2k}(\tau)
\end{align*}
as claimed.
\end{proof}

\begin{remark}
We should note that it is in that last step that we used the fact that \(k > 1\). Strictly speaking we should write\footnote{Note that in this double sum we should exclude the \((m,n) = (0, 0)\) term}
\[
G_{2k}(\tau) = \sum_{m=-\infty}^\infty\sum_{n=-\infty}^\infty \frac{1}{(m\tau + n)^{2k}}
\]
and so the operations done to compare \(G_{2k}(-1/\tau)\) and \(t^{2k}G_{2k}(\tau)\) require us to reorder the summation. This can only be done if the sum is absolutely convergent, which only occurs if \(k > 1\). In fact, we will show later that \(G_2(\tau)\) is {\em not} modular, but that it satisfies a modified transformation law which makes it {\em quasi-modular}.
\end{remark}

Since all of the functions \(G_{2k}\) are invariant under the operation \(\tau \mapsto \tau + 1\), it follows that we can expand them as a series in \(q = e^{2\pi i \tau}\). The resulting Fourier coefficients are often very nice arithmetic functions, and it is from this perspective that we can consider modular forms to be a particularly nice class of generating functions.

We will begin by computing the Fourier expansion of the Eisenstein series. It should be noted that this is valid for \(k \geq 1\), and not just for \(k > 1\).

\begin{proposition}\label{prop_fourier}
The Fourier expansion of \(G_{2k}(\tau)\) is given by
\[
G_{2k}(\tau) = 2\zeta(2k) + \frac{2(-1)^k(2\pi)^{2k}}{(2k-1)!}\sum_{d=1}^\infty \Big(\sum_{k \mid d} k^{2k-1}\Big) q^d.
\]
\end{proposition}

\begin{proof}
We begin by considering the auxiliary function
\[
f(\tau) = \frac{1}{\tau} + \sum_{n=1}^\infty \Big(\frac{1}{\tau + n} + \frac{1}{\tau - n}\Big).
\]
It is not hard to show (by considering its poles and residues) that we have the equality
\begin{equation}\label{eq_atan}
f(\tau) = \pi  \arctan(\pi \tau) = \pi i \frac{e^{2\pi i \tau} + 1}{e^{2\pi i \tau} - 1}.
\end{equation}

We now compute its \((2k-1)\)-st derivative to find that
\[
f^{(2k-1)}(\tau) = -(2k-1)!\sum_{n \in \ZZ} \frac{1}{(\tau + n)^{2k}}.
\]
Moreover, we have from the definition of \(G_{2k}(\tau)\) that
\[
G_{2k}(\tau) = 2\zeta(2k) - \frac{2}{(2k-1)!}\sum_{m=1}^\infty f^{(2k-1)}(m\tau).
\]
We should note here that it is important that \(2k\) be even in this case so that
\[
\frac{1}{(m\tau + n)^{2k}} = \frac{1}{\big((-m)\tau + (-n)\big)^{2k}}
\]
so that we obtain all of the terms in \(G_{2k}(\tau)\) from this method.

Now, by \eqref{eq_atan}, we have that 
\[
f(\tau) = -\pi i(1 + 2q + 2q^2 + 2q^3 + \cdots)
\]
where \(q = e^{2\pi i \tau}\). Moreover, as usual we have that
\[
\frac{d}{d\tau} = 2\pi i q \frac{d}{dq}
\]
and so it follows that
\begin{align*}
f^{(2k - 1)}(\tau) &= -2\pi i (2\pi i)^{2k-1} \sum_{d=1}^\infty d^{2k-1}q^d \\
 &= -(-1)^k (2\pi)^{2k} \sum_{d=1}^\infty d^{2k-1}q^d.
\end{align*}
Thus we conclude that
\begin{align*}
G_{2k}(\tau) &= 2\zeta(2k) + \frac{2(-1)^k(2\pi)^{2k}}{(2k-1)!}\sum_{m=1}^\infty\sum_{d=1}^\infty d^{2k-1}q^{md} \\
&= 2\zeta(2k) + \frac{2(-1)^k(2\pi)^{2k}}{(2k-1)!}\sum_{d=1}^\infty \Big(\sum_{k \mid d} k^{2k-1}\Big) q^d
\end{align*}
as claimed.
\end{proof}

\begin{remark}
We should also note that there are a variety of normalizations that are used for Eisenstein series. As we saw in Proposition \ref{prop_fourier}, \(G_{2k}(\tau)\) is given by
\[
G_{2k}(\tau) = 2\zeta(2k) + \mathcal{O}(\tau).
\]
We will often consider the alternative normalization
\[
E_{2k}(\tau) = \frac{G_{2k}(\tau)}{2\zeta(2k)} = 1 + \mathcal{O}(\tau).
\]
In particular, we have that
\begin{gather*}
E_2(\tau) = 1 - 24 \sum_{d=1}^\infty \sigma_1(d) q^d,\\
E_4(\tau) = 1 + 240 \sum_{d=1}^\infty \sigma_3(d) q^d, \\
E_6(\tau) = 1 - 504 \sum_{d=1}^\infty \sigma_5(d) q^d
\end{gather*}
where \(\sigma_r(d) = \sum_{k \mid d} k^{r-1}\).
\end{remark}

Another modular form of particular interest comes from noting that \(E_4(\tau)^3\) and \(E_6(\tau)^2\) are both modular of weight 12. Since they both have the same constant term, it follows that the form
\[
\Delta(\tau) = \frac{E_4(\tau)^3 - E_6(\tau)^2}{1728} = q + \cO(q^2)
\]
is a {\em cusp form}. Due to its special role which we will see later, it gets a special name.

\begin{definition}
We define the \idx{modular discriminant} to be the function
\[
\Delta(\tau) = \frac{E_4(\tau)^3 - E_6(\tau)^2}{1728}.
\]
\end{definition}

Its Fourier expansion reads
\[
\Delta(\tau) = q - 24q^2 + 252q^3 - 1472q^4 + 4830q^5 - 6048q^6 - 16744q^7 + \cdots,
\]
and it turns out that it can be expressed as
\[
\Delta(\tau) = q\prod_{k=1}^\infty(1 - q^k)^{24} = \eta(\tau)^{24}
\]
where \(\eta(\tau) = q^{1/24}\prod_{k=1}^\infty (1 - q^k)\)

\begin{definition}
A \idx{cusp form} is a modular form whose Fourier expansion at each of the cusps has no constant term.
\end{definition}

Here are some basic properties of modular forms.

\begin{definition}
Define the vector space \(\cM_k(\Gamma)\) to be
\[
\cM_k(\Gamma) = \{f : \HH \to \CC \mid f \text{ is modular of weight } k\},
\]
and also define
\[
\cM_*(\Gamma) = \bigoplus_{k=0}^\infty \cM_k(\Gamma).
\]
Finally, we define the space of cusp forms to be
\[
\cS_k(\Gamma) = \{f : \HH \to \CC \mid f \text{ is a cusp form of weight } k\}.
\]
\end{definition}

\begin{proposition}
The vector space \(\cM_*(\Gamma)\) is a graded \(\CC\)-algebra.
\end{proposition}

In fact, it is a finitely generated \(\CC\)-algebra, which we will prove later. In essence, this is because the quotient \([\HH/\Gamma]\) is a compact orbifold curve; as such, sections of bundles over this stack form a finitely generated algebra.

We can say even more in the case that \(\Gamma = PSL_2\ZZ\), which points out the special r\^ole that \(\Delta(\tau)\) plays. We leave the proof as an excercise.

\begin{proposition}\label{prop_cusp_isom}
Multiplication by \(\Delta(\tau)\) induces an isomorphism
\[
\cM_k(PSL_2\ZZ) \to \cS_{k+12}(PSL_2\ZZ).
\]
\end{proposition}

We also have the following, which can be generalized to groups other than \(PSL_2\ZZ\) with some care.

\begin{proposition}\label{prop_low_dim}
For even \(0 \leq k < 12\), the dimensions of the spaces \(\cM_k(PSL_2\ZZ)\) are given by
\begin{center}
\begin{tabular}{c|cccccc}
$k$ & $0$ & $2$ & $4$ & $6$ & $8$ & $10$\\
\hline
\(\dim \cM_k(PSL_2\ZZ)\) & $1$ & $0$ & $1$ & $1$ & $1$ & $1$
\end{tabular}
\end{center}
\end{proposition}

\begin{proof}
We will provide a sketch of a proof. A good reference is \cite{gunning_modular}. The idea is to consider the quotient \(\HH/PSL_2\ZZ\), which is an orbifold Riemann surface. If we add in the cusp at infinity, then it is compact, and so we can use the Riemann-Roch theorem to compute the dimension of the space of sections of a line bundle on the quotient.

The only concern that we have to consider is that there is some stacky structure at the fixed points. However, we can surpass this by lifting the functions on the quotient to the cover, and paying particular attention to the order of vanishing of functions at those stacky points in the resulting quotient.
\end{proof}

\begin{remark}
This can be extended to other groups \(\Gamma\), but it is simplest to state for \(\Gamma = PSL_2\ZZ\).
\end{remark}

From these two propositions we are able to derive the following nice description of \(\cM_*(PSL_2\ZZ)\).

\begin{theorem}
We have the isomorphism of graded algebras
\[
\cM_*(PSL_2\ZZ) \cong \CC[G_4, G_6] \cong \CC[E_4, E_6].
\]
\end{theorem}

\begin{proof}
What we prove is that the monomials \(E_4(\tau)^\alpha E_6(\tau)^\beta\) form a basis of the space of modular forms.

For \(0 \leq k < 12\), this is clear by Proposition \ref{prop_low_dim}. So it suffices to show that this is true for \(k \geq 12\). So let \(k \geq 12\) be an even integer. In such a case, we can choose non-negative integers \(\alpha, \beta\) such that \(4\alpha + 6\beta = k\). So let \(f(\tau) \in \cM_k(PSL_2\ZZ)\), and write
\[
f(\tau) = \sum_{d=0}^\infty a_dq^d
\]
where \(q = e^{2\pi i \tau}\). Then it follows that
\[
f(\tau) - a_0E_4(\tau)^\alpha E_6(\tau)^\beta
\]
is a cusp form of weight \(k \geq 12\). From Proposition \ref{prop_cusp_isom}, it follows that we can write
\[
f(\tau) - a_0E_4(\tau)^\alpha E_6(\tau)^\beta = \Delta(\tau)g(\tau)
\]
where \(g(\tau) \in \cM_{k-12}(PSL_2\ZZ)\). Since the terms \(E_4(\tau), E_6(\tau)\) are algebraically independent (Exercise \ref{exer_alg_ind}), the claim follows from induction.
\end{proof}

This has a number of nice consequences. We can check, for example, that the dimensions of the homogeneous pieces are given by

\begin{center}
\begin{tabular}{c | c c c c c c c c c}
\(k\) & 0 & 2 & 4 & 6 & 8 & 10 & 12 & 14 & 16\\
\hline
\(\dim \cM_k(PSL_2\ZZ)\) & 1 & 0 & 1 & 1 & 1 & 1 & 2 & 1 & 2
\end{tabular}
\end{center}

In particular, we note that \(\dim \cM_8(PSL_2\ZZ) = 1\). This is of course generated by \(E_8(\tau)\). However, it also contains the form \(E_4(\tau)^2\)! It follows that we must have that \(E_8(\tau)\) is a multiple of \(E_4(\tau)^2\). Since they have the same constant term\footnote{This is why we choose this particular normalization}, they must be equal. This gives us some surprising equalities of sum-of-divisors functions, including the following proposition.

\begin{proposition} \label{prop_divisors}
We have the equality
\[
\sigma_7(d) = \sigma_3(d) + 120 \sum_{m=1}^{d-1} \sigma_3(m)\sigma_3(d-m)
\]
where as before, \(\sigma_r(d) = \sum_{k \mid d} k^r\).
\end{proposition}

\begin{proof}
This follows directly from the equality \(E_8(\tau) = E_4(\tau)^2\).
\end{proof}

\begin{remark}
Perhaps more interestingly, consider the following. Let \(\Lambda\) be a rank \(r\)  lattice, and let
\[
\theta_\Lambda(q) = \sum_{\lambda \in \Lambda} q^{\frac{1}{2}\lambda \cdot \lambda}
\]
be the \idx{theta function} of \(\Lambda\). For example, if \(\Lambda = \ZZ \subset \RR\), then we would have that
\[
\theta_\ZZ = 1 + 2q^{1/2} + 2q + 2q^{3/2} + \cdots
\]
which is one of the classic Jacobi theta functions. It turns out that if \(\Lambda\) is a \idx{unimodular lattice} (i.e. one for which the intersection form has determinant 1), then the the theta function \(\theta_\Lambda(e^{2\pi i \tau})\) is a modular form\footnote{This actually requires an extension of the notion of modularity to (a) deal with characters of the group \(\Gamma\) and (b) to deal with forms of half-integer wieght. However, for the case at hand (\(r = 8\)), no such generalization is needed.} of weight \(\frac{r}{2}\). In particular, if we consider the \(E_8\) lattice, whose Dynkin diagram is given by
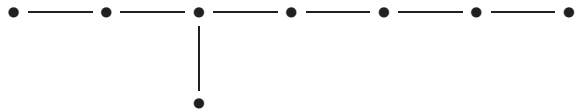
\[
\xymatrix{
\bullet \ar@{-}[r] & \bullet \ar@{-}[r] &\bullet \ar@{-}[r]\ar@{-}[d] &\bullet \ar@{-}[r] &\bullet \ar@{-}[r] &\bullet \ar@{-}[r] &\bullet \\
&& \bullet
}
\]
then this is a modular form of weight 4. That is, it must be a multiple of \(E_4(\tau)\). In particular, since the two functions have constant term 1, we find that
\[
\theta_{E_8}(e^{2\pi i \tau}) = E_4(\tau)
\]
and in particular, the number of elements of norm 2 in \(E_8\) is 240.
\end{remark}

Similar to the proof of Proposition \ref{prop_divisors} we also have, since \(\dim \cM_{10}(PSL_2\ZZ) = 1\), that \(E_4(\tau)E_6(\tau) = E_{10}(\tau)\). The first case without this situation is \(\cM_{12}(PSL_2\ZZ)\), which as we saw earlier yields the discriminant \(\Delta(\tau)\).

\section{Quasi-modular forms}

The theory of modular forms is very beautiful, but one could argue that it suffers from one blemish. After all, \(G_2(\tau)\) has just as nice a definition and Fourier expansion as the higher weight Eisenstein series, but it is not modular. This feels like a bit of a shortcoming in this theory.

Furthermore, many nice situations in which generating functions of this type arise involve \(E_2(\tau)\)! Thus it would be nice to have a framework which includes all of the Eisenstein series.

To do so, we note first the following Proposition.

\begin{proposition}
The function \(E_2(\tau)\) satisfies the following transformation law:
\[
(c\tau + d)^{-2}E_2(\gamma \tau) = E_2(\tau) + 12\frac{c}{c\tau + d}.
\]
\end{proposition}

\begin{proof}
We note from Exercise \ref{exer_log_deriv} that we have the equality
\[
E_2(\tau) = \frac{d}{d\tau}\log\Delta(\tau).
\]
Consider now the fact that \(\Delta(\tau)\) is weight 12 modular. That is, \(\Delta(\gamma \tau) = (c\tau + d)^{12}\Delta(\tau)\). If we then proceed to differentiate both sides of this equation (noting that \(\frac{d}{d\tau}\frac{a\tau + b}{c\tau + d} = \frac{1}{(c\tau + d)^2}\)), we find that
\[
(c\tau + d)^{-2}\Delta'(\gamma\tau) = (c\tau + d)^{12}\Delta(\tau) + 12c(c\tau + d)^{11}\Delta(\tau).
\]
If we then divide both sides of the equation by \(\Delta(\gamma\tau)\), we have that
\[
(c\tau + d)^{-2} E_2(\gamma\tau) = E_2(\tau) + 12\frac{c}{c\tau + d}
\]
as claimed.
\end{proof}

What we can conclude from this is that even though \(E_2(\tau)\) is not modular, it does satisfy a modified transformation law, which we will use to motivate our definition of quasi-modular forms.

\begin{remark}
It is worth noting that this anomalous transformation law can be derived from a careful attention to what happens in the non-absolutely convergent series
\[
G_2(\tau) = \sum_{(m, n) \neq (0,0)} \frac{1}{(m\tau + n)^2}
\]
when we reverse the order of summation after considering the modular transformation \(G_2(-1/\tau)\).
\end{remark}

\begin{definition}\label{qmod_A}
A \idx{quasi-modular form} of weight \(k\) and depth (at most) \(r\) for the group \(\Gamma\) is a holomorphic function \(f: \HH \to \CC\) such that, for all \(\gamma = \begin{pmatrix}a & b \\ c & d \end{pmatrix} \in \Gamma\), 
\[
(c\tau + d)^{-k} f(\gamma \tau)
\]
can be written as a polynomial of degree (at most) \(r\) in \(\frac{c}{c\tau + d}\) with holomorphic coefficients. That is,
\[
(c\tau + d)^{-k} f(\gamma \tau) = \sum_{m=0}^r f_m(\tau) \Big(\frac{c}{c\tau+d}\Big)^m
\]
for some holomorphic functions \(f_m(\tau)\).
\end{definition}

It follows immediately that any modular form is quasi-modular (of depth 0), and that \(E_2(\tau)\) is quasi-modular of weight 2 and depth 1.

\begin{remark}
By considering the matrix \(\gamma = \begin{pmatrix}1 & 0 \\ 0 & 1 \end{pmatrix}\), we see immediately that \(f_0(\tau) = f(\tau)\).

We can actually say a little more (which is left as an exercise). Given  \(f(\tau)\) and its companions, \(f_0(\tau), \ldots, f_r(\tau)\), it can be shown that each of the functions \(f_m(\tau)\) is a modular form of weight \(k - 2m\) for the same group \(\Gamma\).
\end{remark}

Now, since we have expanded our definition of modularity, we should be cautious: after all, it is possible that we have expanded it too far. Perhaps every function transforms in this way? It is natural to ask what is the relationship between the \(\CC\)-algebra of quasi-modular forms and the \(\CC\)-algebra of modular forms.

Fortunately, it turns out that this is just the right generalization, as we see in the following theorem.

\begin{theorem}
Let \(\cM_*(\Gamma)\) denote the graded \(\CC\)-algebra of modular forms for the group \(\Gamma\), and let \(\cQM_*(\Gamma)\) denote the graded \(\CC\)-algebra of quasi-modular forms for the same group. Then we have an isomorphism of graded \(\CC\)-algebras
\[
\cQM_*(\Gamma) \cong \cM_*(\Gamma) \otimes_\CC \CC[G_2(\tau)]
\]
(where of course, the weight of \(G_2(\tau)\) is 2).
\end{theorem}

\begin{proof}
The proof of this is surprisingly simple. We simply induct on the depth of the form.

Let \(f(\tau)\) be quasi-modular of weight \(k\) and depth \(r\). We claim that the function
\[
F(\tau) = f(\tau) - f_r(\tau)\Big(\frac{E_2(\tau)}{12}\Big)^r
\]
is quasi-modular of weight \(k\) and depth strictly less than \(r\), from which the conclusion will follow.

Let us consider \((c\tau + d)^{-k}F(\gamma\tau)\). Then we find that
\begin{align*}
(c\tau+d)^{-k}F(\gamma\tau)
&= (c\tau + d)^{-k}\Big(f(\gamma\tau) - f_r(\gamma\tau)\Big(\frac{E_2(\gamma\tau)}{12}\Big)^r\Big) \\
&= (c\tau + d)^{-k}f(\gamma\tau) - \frac{1}{12^r}(c \tau + d)^{-k+2r}f_r(\gamma\tau)\big((c\tau + d)E_2(\gamma\tau)\big)^r
\end{align*}
which, due to the (quasi-)modularity of each of the factors, becomes
\begin{align*}
(c\tau+d)^{-k}F(\gamma\tau)
&= \sum_{m=0}^r f_m(\tau)\Big(\frac{c}{c\tau+d}\Big)^m - \frac{1}{12^r}f_r(\tau)\Big(E_2(\tau) + \frac{12c}{c\tau+d}\Big)^r\\
&= \sum_{m=0}^r f_m(\tau)\Big(\frac{c}{c\tau+d}\Big)^m \\
&\quad -\frac{1}{12^r}f_r(\tau) \bigg(\sum_{m=0}^r {r \choose {r-m}}E_2(\tau)^{r-m}\Big(\frac{12c}{c\tau+d}\Big)^{m}\bigg).
\end{align*}
However, we see that the \(m = r\) terms of both of these are exactly
\[
f_r(\tau)\Big(\frac{c}{c\tau+d}\Big)^r
\]
which cancel out, whence the claim.
\end{proof}

This leads us to the special case for \(\Gamma = PSL_2\ZZ\).

\begin{corollary}
The \(\CC\)-algebra of quasi-modular forms for \(PSL_2\ZZ\) is
\[
\cQM_*(PSL_2\ZZ) \cong \CC[G_2, G_4, G_6].
\]
\end{corollary}

One other nice property that this algebra has is essentially due to S. Ramanujan. He noted that

\begin{proposition}\label{prop_derivatives_eisenstein}
Let \(D = \frac{1}{2\pi i} \frac{d}{d\tau} = q\frac{d}{dq}\). Then we have the following relations amongst the Eisenstein series.
\begin{gather*}
DE_2(\tau) = \frac{E_2(\tau)^2 - E_4(\tau)}{12}.\\
DE_4(\tau) = \frac{E_2(\tau)E_4(\tau) - E_6(\tau)}{3}. \\
DE_6(\tau) = \frac{E_2(\tau)E_6(\tau) - E_4(\tau)^2}{2}.
\end{gather*}
\end{proposition}

We will provide a proof of this in just a moment.

To generalize this, we first note that modular forms are {\em almost} closed under differentiation. That is, if we consider a modular form \(f(\tau)\) of weight \(k\), then it satisfies
\[
f(\gamma \tau) = (c\tau + d)^k f(\tau).
\]
If we differentiate both sides of this equation we obtain
\[
(c\tau + d)^{-2}f'(\gamma \tau) = (c \tau + d)^k f'(\tau) + kc(c \tau + d)^{k-1}f(\tau)
\]
or more clearly,
\[
f'(\gamma\tau) = (c\tau + d)^{k+2}f'(\tau) + k (c \tau + d)^{k+2} f(\tau)\frac{c}{c\tau + d}.
\]
It follows then that the derivative is a quasi-modular form of weight \(k+2\) (and depth one greater)!

More generally, we have the following.

\begin{theorem}
The operation of differentiation acts on quasi-modular forms via
\[
\frac{d}{d\tau} : \cQM_{k}(\Gamma) \hookrightarrow \cQM_{k+2}(\Gamma).
\]
Thus, the ring of quasi-modular forms is closed under differentiation.
\end{theorem}

We are now in a position to easily prove Proposition \ref{prop_derivatives_eisenstein} in a way that showcases the utility of modularity.

\begin{proof}
Let us prove the first of these equalities. Since we have that \(\cQM_*(PSL_2\ZZ) \cong \CC[G_2, G_4, G_6]\) as graded \(\CC\)-algebras, it follows that the dimension of the degree 4 piece is 2, spanned by \(E_2(\tau)^2\) and \(E_4(\tau)\). Since \(DE_2(\tau)\) is a weight 4 (by the above Theorem), we can simply compare the first two Fourier coefficients of both sides. We have on the left-hand side,
\[
D(1 - 24q + \mathcal{O}(q^2)) = -24q + \mathcal{O}(q^2)
\]
and on the right-hand side
\[
\frac{\big(1 - 24q + \mathcal{O}(q^2)\big)^2 - \big(1 + 240q + \mathcal{O}(q^2)\big)}{12} = \frac{-288q + \mathcal{O}(q^2)}{12} = -24q + \mathcal{O}(q^2)
\]
as claimed.
\end{proof}

\section{Applications of Modular forms}

To apply (quasi-)modular forms to the context of Mirror Symmetry, we need to use a slightly different definition of quasi-modular forms, which is to be found in \cite{quasimodular}.

\begin{definition}
Let \(\Gamma\) be a finite index subgroup of \(PSL_2\ZZ\) as in definition \ref{mod_def}. An \idx{almost-holomorphic modular form} of weight \(k\) is a function \(f : \HH \to \CC\) for the group \(\Gamma\) if it satisfies the transformation law
\[
f (\gamma \tau) = (c\tau + d)^kf(\tau)
\]
for all \(\gamma \in \Gamma\), if it is bounded at infinity, and if it is at most polynomial in \(Y^{-1} = (\im \tau)^{-1}\). That is, if we can write it as
\[
f(\tau) = \sum_{m=0}^r f_m(\tau)Y^{-m}
\]
for some \(r \in \NN\).
\end{definition}

\begin{definition}\label{qmod_B}
Let \(\Gamma\) be as before. Then we define a \idx{quasi-modular form} to be the constant part (with respect to \(Y^{-1}\)) of an almost-holomorphic modular form.
\end{definition}

\begin{example}
As \(E_2(\tau)\) was quasi-modular before, it would probably be best that it still be quasi-modular under this alternate definition. Indeed, one can show that
\[
E_2^*(\tau, \bar \tau) = E_2(\tau) - \frac{6i}{\im \tau}
\]
satisfies the modular transformation law (Exercise \ref{alt_quasimod_e2}). Thus its constant term, \(E_2(\tau)\), is a quasi-modular form by this definition.
\end{example}

\begin{remark}
It is not too hard to show that these two definitions of quasi-modularity are equivalent. The gist of it is that in each case, the only new functions that we introduce are powers of the Eisenstein series \(E_2(\tau)\), which will be shown in the exercises.
\end{remark}

\begin{remark}
Given an almost-holomorphic modular form \(f(\tau, \bar \tau)\), we can equally obtain the associated quasi-modular form by the limit
\[
f(\tau) = \lim_{\bar \tau \to \infty} f(\tau, \bar \tau)
\]
which is a perspective that will be relevant shortly.
\end{remark}

Our main application of interest is that (quasi-)modular forms arise naturally in studying the enumerative geometry of Calabi-Yau manifolds through the use of Mirror Symmetry. Let us expand on this idea.

\begin{remark}
From here onwards, for simplicity of notation we will always mean quasi-modular whenever we write modular, as they are the objects that arise naturally in Mirror Symmetry.
\end{remark}

\begin{definition}
A \idx{Calabi-Yau} manifold is a K\"ahler manifold \(X\) such that \(K_X \cong \cO_X\).
\end{definition}

The main example that we will consider for the time being is a complex torus, or an elliptic curve; note that as the tangent bundle of an elliptic curve is trivial, we also have that \(K_E \cong \cO_E\).

\begin{remark}
In different contexts, there may be further requirements for a manifold to be Calabi-Yau. For example, some authors require
\begin{itemize}
\item \(\pi_1(X)\) is trivial.
\item \(h^{1,0}(X) = 0\).
\item \(h^{p,0}(X) = 0\) for \(0 < p < \dim X\).
\end{itemize}
Note that for a Calabi-Yau threefold, the first condition implies the second, which is equivalent to the third.
\end{remark}

Given a Calabi-Yau manifold (read: elliptic curve), there are two types of moduli which we naturally associate with this manifold. There is the moduli of inequivalent complex structures, which we denote \(\cM_\CC^X\), and the (complexified) K\"ahler moduli, denoted \(\cM_K^X\).

For the case of an elliptic curve, we have the classical fact that \(\cM_\CC^E \cong \HH/PSL_2\ZZ\), while the K\"ahler moduli space is given by
\[
\cM_K^E = \{ \omega \in H^2(X, \CC) \mid \im \omega \text{ is a K\"ahler class}\}/H^2(X, \ZZ)
\]
which in our case is \(\HH/\ZZ\) with multivalued \(\HH\)-coordinate given by
\[
t = \int_E \omega.
\]
We should note that a more natural choice of coordinate on this moduli space will be \(q = e^{2\pi i t}\), due to the multivalued nature of this coordinate.

One of the main claims of Mirror Symmetry is that, given a Calabi-Yau manifold \(X\), there is a second Calabi-Yau manifold \(\hat{X}\), the {\em mirror manifold} of \(X\) together with local isomorphisms---the so-called ``mirror maps''---centered around certain special points of moduli
\begin{gather*}
\phi_X : \cM_K^X \to \cM_\CC^{\hat{X}} \\
\phi_{\hat{X}} : \cM_K^{\hat{X}} \to \cM_\CC^X 
\end{gather*}
such that certain functions defined on one space can be computed on the other space.

\begin{remark}
For an elliptic curve, the mirror must be another elliptic curve, since that is the only Calabi-Yau manifold of dimension 1. Moreover, the mirror maps can be shown to be the very simple
\[
\phi_E : t \mapsto \tau.
\]
\end{remark}

\begin{remark}
For a mirror pair \((X, \hat X)\) of Calabi-Yau threefolds, it follows that we have the following equality of Hodge numbers.
\[
h^{2,1}(X) = h^{1,1}(\hat{X}) \qquad \qquad h^{2,1}(\hat{X}) = h^{1,1}(X)
\]
since \(\dim\cM_\CC^X = h^1(X, T_X) = h^{3-1,1}(X)\) and \(\dim \cM_K^X = h^1(X, \Omega_X)\).
\end{remark}

The important part of Mirror Symmetry is understanding the statement that ``certain functions can be computed on the other space''. Let us focus on the elliptic curve and see what this means.

We begin by recalling the following important function.

\begin{definition}
We define the Weierstrass \(\wp\)-function to be
\[
\wp(z,\tau) = \frac{1}{z^2} + \sum_{(m,n) \neq (0,0)} \bigg(\frac{1}{\big(z - (m + n\tau)\big)^2} - \frac{1}{(m+n\tau)^2}\bigg).
\]
\end{definition}

From string-theoretic considerations, we consider the function of the complex modulus \(\tau\) i.e. on \(\cM_\CC^E\) (for \(g \geq 2\)) given by
\[
F_g^B(\tau) = \sum_{\Gamma}\frac{I_\Gamma}{|\Aut\Gamma|}
\]
where we sum over trivalent, genus \(g\) graphs (i.e. such that \(h^1(\Gamma) = g\)). These graphs have \(2g-2\) vertices \(v_i\) and \(3g-3\) edges \(e_i\). Moreover, if we define the function \(P(z,\tau)\) to be\footnote{Note that we are using the non-holomorphic extension of \(E_2(\tau)\) so that this is well-defined on \(\cM_\CC^E\)}
\[
P(z,\tau) = \frac{1}{4\pi^2}\wp(z,\tau) + \frac{1}{12}E_2^*(\tau)
\]
(for \(z \neq 0\), at least), then we define the weighting \(I_\Gamma\) to be
\[
I_\Gamma = \idotsint \prod_e \big(-P(z_{e_{v_1}} - z_{e_{v_2}},\tau)\big) \prod_v dz_v 
\]
where 
\begin{itemize}
\item we take the product over the \(3g - 3\) edges \(e\), whose end vertices are \(e_{v_1}\) and \(e_{v_2}\),
\item we have a variable \(z_v\) associated to each vertex \(v\),
\item we integrate along \(2g-2\) non-intersecting loops \(z_v \to z_v + 1\) in \(E\).
\end{itemize}

\begin{example}
Let us consider the graph
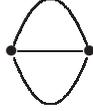
\[
\xymatrix{
\\
*{\bullet} \ar@{-}[r]\ar@/^2em/@{-}[r]\ar@/_2em/@{-}[r] & *{\bullet} \\ \\
}
\]
(the so-called \(\theta\)-graph) which is a genus 2 graph, with two vertices and three edges. It follows that our function of interest is
\[
I_\Gamma =  \iint \big(-P(z_1 - z_2,\tau)\big)^3dz_1dz_2
\]
which can be computed to be (see \cite{roth_yui_B}) to be 
\[
\frac{1}{2^8 3^4 5}\big(10E_2^*(\tau)^3 - 6E_2^*(\tau)E_4(\tau) - 4E_6(\tau)\big)
\]
which is clearly quasi-modular in the limit \(\bar \tau \to \infty\). More generally, these functions will be all quasi-modular due to the Laurent expansion of the Weierstrass \(\wp\)-function, which gives us that
\[
P(z,\tau) = -\frac{1}{(2\pi i z)^2} - \sum_{n =1}^\infty \zeta(1-2n) \frac{E_{2n}(\tau)}{(2n-2)!}(2\pi i z)^{2k-2}.
\]
\end{example}

\begin{example}
We can further show that the only relevant genus 3 graphs are the following:
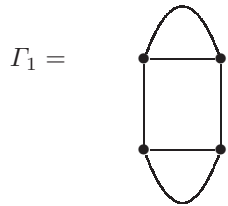
\[
\xymatrix{
\\
\Gamma_1 = & *{\bullet} \ar@{-}[r]\ar@{-}[d]\ar@/^2em/@{-}[r] & *{\bullet} \ar@{-}[d]\\
&*{\bullet} \ar@{-}[r]\ar@/_2em/@{-}[r] & *{\bullet} \\ \\
}
\]
and
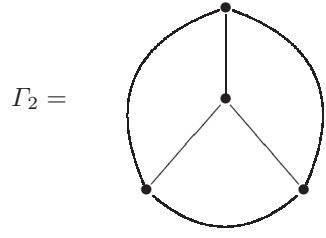
\[
\xymatrix{
& & *{\bullet} \ar@{-}[d] \\
\Gamma_2 = & & *{\bullet} \\
& *{\bullet} \ar@{-}[ur] \ar@/^2.1em/@{-}[uur] \ar@/_1.4em/@{-}[rr] & & *{\bullet} \ar@{-}[ul]\ar@/_2.1em/@{-}[uul] \\ \\
}
\]
which produce the quasimodular forms
\[
I_{\Gamma_1} = \frac{1}{2^7 3^6}(4E_6^2 + 4E_4^3 - 12E_2E_4E_6 - 3E_2^2E_4^2 + 4E_2^3E_6 + 6E_4^2E_4 - 3E_2^6)
\]
and
\[
I_{\Gamma_2} = \frac{1}{2^8 3^4}(E_4-E_2^2)^3
\]
respectively.
\end{example}

On the K\"ahler (A-model) side, we have the following generating function, which contains enumerative information about the elliptic curve.

\begin{definition}
We begin by defining \(N_{d,g}\) to be the number of degree \(d\), genus \(g\) covers of an elliptic curve which are simply ramified at \(2g- 2\) points, weighted by automorphisms of the cover. For more information on this, see \cite{trop_hurwitz, roth_yui_A}.

Now, for \(g \geq 2\), we define the genus \(g\), A-model generating function to be
\[
F_g^A(t) = \sum_{d=1}^\infty N_{d,g} q^d
\]
where \(q = e^{2\pi i t}\). Note that this is a function of the K\"ahler modulus \(t = \int_E\omega\).
\end{definition}

\begin{theorem}[Mirror Symmetry for the Elliptic curve, see \cite{dijk, trop_elliptic}]
The generating functions \(F_g^A(t)\) can be computed as
\[
F_g^A(t) = \lim_{\bar t \to \infty} F_g^B(t,\bar t)
\]
\end{theorem}

The numbers \(N_{d,g}\) are in general rather difficult to compute, and this theorem tells us that we can compute them by computing certain integrals on the mirror elliptic curve, which is rather surprising. Moreover, it tells us that the generating functions \(F_g^A(t)\) are quasi-modular forms, which is by no means obvious from the definition.

What happens in the general case? The limit \(\bar \tau \to \infty\) or \(\bar t \to \infty\) correspond to the so-called {\em large complex structure limit} or {\em large volume limit} in the moduli spaces \(\cM_\CC^E\) and \(\cM_K^E\), respectively. These are the ``special points of moduli'' discussed earlier. In particular, what we expect for a general Calabi-Yau \(X\) with mirror \(\hat X\) is the same general picture. Specifically, we should have the following:

The function \(F_g^{\hat X, B}(\tau)\) is a (non-holomorphic) section of a line bundle \(L^{2g-2}\) on the moduli space \(\cM_\CC^{\hat X}\), which is to be thought of as a sort of generalized modular form. Mirror Symmetry now claims that we have a local isomorphism \(\phi_X\) around the large volume/large complex structure limit points so that the function \(F_g^{X,A}(t)\) can be computed as
\[
F_g^{X,A}(t) = \lim_{\bar t \to \infty} F_g^{\hat X, B}\big(\phi_X(t)\big)
\]
i.e. that \(F_g^{X, A}(t)\) is a quasi-modular form, according to the above definition.

\section{Further Examples}

Such modular forms arise in higher-dimensional settings as well, and in particular they arise when counting curves on surfaces. In such a case, there is a general framework in \cite{gottsche} for enumerating curves in a fixed linear systems subject to some point constraints (i.e. if the linear system is \(d\)-dimensional, then we demand that our curves pass through \(d\) generic points, which will (morally) reduce it to a finite count). In the special case that \(K_S \cong \cO_S\)---that is, K3 or Abelian surfaces---then we end up with modular generating functions.

For the purposes of this section, we consider a {\em third} normalization of the Eisenstein series, given by
\[
\widehat{E}_{2k}(\tau) = \frac{(-1)^k(2k-1)!}{2(2\pi)^{2k}}G_{2k}(\tau)
\]
which for \(\widehat{E}_2(\tau)\) reads (and similarly for higher values of \(k\))
\[
\widehat{E}_2(\tau) = -\frac{1}{24} + \sum_{d=1}^\infty \sigma_1(d)q^d,
\]
and where, as usual, \(q = e^{2\pi i \tau}\). With this, we have the following theorems (see \cite{yz, bryan_leung_k3, bryan_leung_ab}).

\begin{theorem}
Let \(N_{d,g}^{K3}\) denote the number of degree \(d\), genus \(g\) curves in a fixed linear system in a generic algebraic K3 surface. Then we have
\[
F_g^{K3}(q) = \sum_{d=0}^\infty N_{d,g}^{K3} q^{d-1} = \frac{\big(D\widehat{E}_2(\tau)\big)^g}{\Delta(\tau)}.
\]
\end{theorem}

\begin{theorem}
Let \(N_{d,g}^A\) denote the number of degree \(d\), genus \(g\) curves in a fixed linear system in a generic algebraic Abelian surface. Then we have
\[
F_g^A(q) = \sum_{d=0}^\infty N_{d,g}^A q^{d+g-1} = \big(D\widehat{E}_2(\tau)\big)^{g-2}D^2\widehat{E}_2(\tau).
\]
\end{theorem}

There is another way to count curves in Abelian surfaces as well. Up to translation there is a \((g-2)\)-dimensional family of geometric genus \(g\) curves in a fixed curve class \(\beta\), which is exactly the codimension of the hyperelliptic locus in \(\overline{\cM}_g\). This suggests that there should be finitely many such curves in a generic Abelian surface, up to translation. Indeed, we find (see \cite{rose}) the following.

\begin{theorem}
Assume the crepant resolution conjecture (see \cite{crc}) for the resolution \(Km(A) \to [A/\pm1]\). Then the generating function for the number of hyperelliptic curves in \(A\) (with some discrete data) can be described explicitly in terms of quasi-modular forms for the group \(\Gamma_0(4)\).
\end{theorem}

One final application of the fact that \(\dim \cM_k(\Gamma)\) is finite is the following. Let \(\Lambda\) be a rank \(r\) lattice, and consider a 1-parameter family of \(\Lambda\)-polarized K3 surfaces \(\pi : X \to C\). This yields a canonical map to the moduli space of \(\Lambda\)-polarized K3 surfaces,
\[
\iota_\pi : C \to \cM_\Lambda.
\]
Within this moduli space, there is a particular class of divisors, the Noether-Lefschetz divisors. These parameterize those K3 surfaces whose Picard rank jumps: That is, we have that
\[
D_{\vec{d},h} = \left\{S \in \cM_\Lambda \text{ such that } \left(\:
\begin{array}{*{13}{c}}
\cline{1-3}
\multicolumn{1}{|c}{ } & & \multicolumn{1}{c|}{ }& d_1 \\
\multicolumn{1}{|c}{ } & \Lambda & \multicolumn{1}{c|}{ }& \vdots \\
\multicolumn{1}{|c}{ } & & \multicolumn{1}{c|}{ }& d_r \\
\cline{1-3}
d_1 & \cdots & d_r & 2h - 2
\end{array}
\right) \hookrightarrow Pic(S)
\right\}
\]
(with the embedding required to be primitive) weighted by certain multiplicity data. For more detail, see \cite{kmps_nlyz}.

The intersection of the image \(\iota_\pi(C)\) with these divisors defines the \idx{Noether-Lefschetz numbers}. That is,
\[
NL_{\vec{d},h}^\pi = \int_{\iota_\pi(C)}D_{\vec{d},h} = \int_C \iota_\pi^*D_{\vec{d},h}.
\]
It follows then from \cite{bor_gkz} that these are the coefficients of a modular form of weight \(\frac{22-r}{2}\) for some group \(\Gamma\). That is, we can determine all of these intersection numbers by looking at finitely many of them, due to the finite dimensionality of \(\cM_k(\Gamma)\).

Another source of modularity comes from certain elliptically fibred Calabi-Yau threefolds. Let us consider the simplest case; for more details, see \cite{rose_yui}.

Let \(\FF_1\) be the first Hirzebruch surface; that is, if we define the \(n\)-th Hirzebruch surface as
\[
\FF_n = \PP(\cO \oplus \cO(n))
\]
then we are interested in \(\FF_1\). We should note that in this case, that there is an alternative description of this surface as the blowup of \(\PP^2\) at a point, which is a del Pezzo surface of degree 8. This can be seen via their toric descriptions:
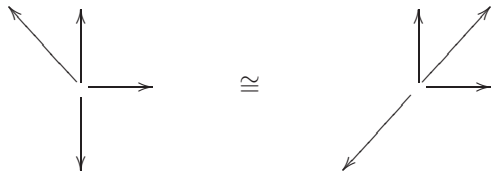
\[
\xymatrix{
 & & & & & &\\
 & \ar[ul]\ar[u]\ar[r]\ar[d] & & \cong& & \ar[dl]\ar[u]\ar[ur]\ar[r] &\\
 & & & & & &
}
\]
In this case, we can construct a certain elliptically fibred threefold \(X\) over \(\FF_1\). Its cone of effective curves is generated by \(C, F, E\), where \(C, F\) are the section and fibre curves of \(\FF_1\), and \(E\) is the class of an elliptic fibre.

Choose now some \(\beta \in H_2(\FF_1)\), and consider the generating function
\[
F_\beta(q) = \sum_{n=1}^\infty N_{\beta + n E}^X q^n
\]
where \(N_{\beta + nE}^X\) is the Gromov-Witten invariant of \(X\) in the class \(\beta + nE\), and in this case we regard \(q\) as a formal variable.

We have then the following result.

\begin{theorem}
The generating functions \(F_C(q)\) and \(F_F(q)\) are given by
\begin{gather*}
F_C(q) = q^{1/2}\frac{E_4(q)}{\eta(q)^{12}}, \\
F_F(q) = -2\frac{E_{10}(q)}{\eta(q)^{24}}
\end{gather*}
and as such, if we let \(q = e^{2\pi i \tau}\), then they are meromorphic modular forms\footnote{There is of course the factor of \(q^{1/2}\) in the first term which does break modularity. However, we can easily include this into the definition of the function, and end up with a modular form as we desire.} of weight \(-2\) for the group \(SL_2\ZZ\).
\end{theorem}

In fact, the same is true of the generating function \(F_{mF}(q)\), although the explicit formula is a little more complicated.

It seems then natural to conjecture the following (which agrees with conjectural formul\ae\ arising from string-theoretic considerations, see \cite{KMW,KMV,AS}).

\begin{conjecture}
The generating function \(F_\beta(q)\) are meromorphic modular forms of weight \(-2\) (for some congruence subgroup \(\Gamma\)) for all \(\beta \in H_2(\FF_1)\).
\end{conjecture}

\section{Exercises}

\begin{enumerate}
\item Show that a map \(f: E_1 \to E_2\) of elliptic curves is determined by a map on the underlying lattices.

\item Show that, up to similarity, a \(2 \times 2\) integer matrix of determinant \(d\) can be written as
\[
\begin{pmatrix}
m & r \\ 0 & n
\end{pmatrix}
\]
with \(mn = d\) and \(0 \leq r < m\).

\item Prove that the map
\[
\cM_k(PSL_2\ZZ) \to \cS_{k+12}(PSL_2\ZZ)
\]
induced by multiplication by \(\Delta(\tau)\) is an isomorphism.

\item\label{exer_alg_ind} Show that the modular forms \(E_4(\tau)\) and \(E_6(\tau)\) are algebraically independent.

\item\label{exer_log_deriv} Show that we can write \(E_2(\tau)\) as the logarithmic derivative of \(\Delta(\tau)\). That is, we have the equality
\[
E_2(\tau) = \frac{d}{d\tau}\log\Delta(\tau).
\]
Hint: It is helpful to remember that \(\Delta(\tau) = q\prod_{k=1}^\infty (1 - q^k)^{24}\)

\item\label{alt_quasimod_e2} Prove that the function
\[
E_2^*(\tau, \overline{\tau}) = E_2(\tau) - \frac{6 i}{\im \tau} = E_2(\tau) + \frac{12}{\tau - \overline{\tau}}
\]
transforms like a modular form of weight 2, although it is not holomorphic.

\item Let \(f(\tau)\) be a quasi-modular form for some group \(\Gamma\). That is, it transforms as
\[
(c\tau + d)^{-k} f(\gamma\tau) = \sum_{m=0}^r f_m(\tau)\Big(\frac{c}{c\tau + d}\Big)^m.
\]
Prove that each of the functions \(f_m(\tau)\) are themselves modular of weight \(k - 2m\).

\item Prove that a genus \(g\) trivalent graph has \(2g-2\) vertices and \(3g-3\) edges.

\end{enumerate}

\bibliographystyle{amsplain}
\bibliography{notes}



\end{document}